\documentclass[11 pt,amstex,3p,times]{article}
\usepackage{hyperref}
\usepackage{enumerate,fullpage}
\usepackage{amssymb,amsmath,graphicx,amsthm}
\usepackage{color}

\newcommand{\be}{\begin{eqnarray*}}
	\newcommand{\en}{\end{eqnarray*}}
\newcommand{\bes}{\begin{eqnarray}}
\newcommand{\ens}{\end{eqnarray}}

\usepackage{epic, curves}
\def\nn{\nonumber}
\newcommand{\al}{\alpha}

\newcommand{\la}{\lambda}

\newcommand{\ep}{\epsilon}

\newtheorem{theorem}{Theorem}[section]

\newtheorem{lemma}{Lemma}[section]

\newtheorem{remark}{Remark}[section]

\def\bq{\begin{equation}}
\def\eq{\end{equation}}
\def\bqq{\begin{eqnarray*}}
	\def\eqq{\end{eqnarray*}}

\def\nn{\nonumber}

\title{\bf Approximation of mild solutions of a semilinear fractional elliptic equation with random noise    }
\author{  Ho Duy Binh $^1$, Erkan Nane $^2$ and Nguyen Huy Tuan $^3$, \footnote{Corresponding author: \url{nguyenhuytuan@tdt.edu.vn }} \\\\
\small $^1$ Faculty of Fundamental Science, Nguyen Hue University, Dong Nai, Viet Nam\\
	\small $^2$ Department of Mathematics and Statistics, Auburn University, Auburn, USA  \\
\small $^{3}$ Applied Analysis Research Group,
Faculty of Mathematics and Statistics,\\
\small Ton Duc Thang University, Ho Chi Minh City, Vietnam\\ \\
}
\begin{document}
	\date{}
	\maketitle
	
	\begin{abstract}
		We study for the first time the  Cauchy problem for semilinear fractional elliptic equation. This paper is concerned with the Gaussian white noise  model  for the initial Cauchy data. We establish the ill-posedness of the problem. Then,  under some  assumption on the exact solution, we propose the Fourier truncation method for stabilizing the
		ill-posed problem. Some convergence rates between the exact solution and the regularized solution is established in $L^2$ and $H^q$ norms. 		
	\end{abstract}


	\section{Introduction}
	The theory of fractional differential equations has received much attention over the past twenty years, since they are
	important in describing the natural models such as diffusion processes, stochastic processes, finance and hydrology. We refer
	for instance to the books \cite{Kilbas,  mark-skorski-12, podlubny-1999, samko}.
	In this paper, we  consider the following Cauchy problem of fractional semi-linear elliptic equations:{ }
	\begin{equation} \label{problem}
		\frac{D^{\beta}\mathbf{u}\left(t,y\right)}{Dt^{\beta}}=\mathcal{A} \mathbf{u}\left(t,y\right)+G\left(t,y,\mathbf{u}\left(t,y\right)\right),\quad\left(t,y\right)\in\Omega:=\Omega_{1}\times\Omega_{2},
	\end{equation}
	associated with the zero Dirichlet boundary condition in $y$ and
	the initial data and nonhomogeneous initial velocity given by
	\begin{equation} \label{condition}
		\mathbf{u}\left(0,y\right)=\mathbf{u}_{0}\left(y\right),\quad\frac{d\mathbf{u}\left(t,y\right)}{dt}\bigg|_{t=0}=\mathbf{u}_{1}\left(y\right),\quad y\in\Omega_{2}.
	\end{equation}
	 In \eqref{problem}, $\beta \in (1,2)$ is the
	 fractional order and $	\frac{D^{\beta}}{Dt^{\beta}}$ denotes the Caputo fractional derivative with respect to $t$,
	 (see \cite{yamamoto,yamamoto1}),
	 \[
	 	\frac{D^{\beta}\mathbf{u}\left(t,y\right)}{Dt^{\beta}}:=\frac{1}{\Gamma(2-\beta)}\int_0^t(t-\eta)^{1-\beta}
	 \frac{\partial^{2} \mathbf{u}}{\partial \eta^{2}} (\eta, y) d\eta,
	 \]
	 where $\Gamma$ is the Gamma function.
The function  $\mathbf{u}:\Omega_{1}\to L^{2}\left(\Omega_{2}\right)$ denotes
	the distribution of a body where $\Omega_{1}:=\left(0,a\right)\subset\mathbb{R}$
	and $\Omega_{2}\subset\mathbb{R}^{n}$ are open, bounded and connected
	domains with a smooth boundary for $n\ge2$ and $a>0$, and $\mathcal{A}$
	is the linear second-order differential operator with variable coefficients
	depending on $y$ only:
	\[
	\mathcal{A}\mathbf{u}\left(t,y\right)=\mathcal{A}_y \mathbf{u}\left(t,y\right)=\sum_{i,j=1}^{n}\frac{\partial}{\partial y_{i}}\left(d_{i,j}\left(y\right)\frac{\partial\mathbf{u}\left(t,y\right)}{\partial y_{j}}\right)+d\left(y\right){\bf u}\left(t,y\right).
	\]
	
	The basic requirement for the coefficients $d_{i,j}\left(y\right)$
	and $d\left(y\right)$ is that $\mathcal{A}$ is a positive, self-adjoint
	operator in the Hilbert space $L^{2}\left(\Omega_{2}\right)$. Consequently,
	there exists an orthonormal basis of $L^{2}\left(\Omega_{2}\right)$,
	denoted by $\left\{ \phi_{p}\right\} _{p\in\mathbb{N}^{*}}$, satisfying
	\begin{equation}
		\phi_{p}\in H_{0}^{1}\left(\Omega_{2}\right)\cap C^{\infty}\left(\overline{\Omega_{2}}\right),\quad\mathcal{A} \phi_{p}\left(y\right)=\lambda_{p}\phi_{p}\left(y\right)\;\text{for}\;y\in\Omega_{2},\label{eq:eigen1}
	\end{equation}
	and the corresponding discrete spectrum $\left\{ \lambda_{p}\right\} _{p\in\mathbb{N}^{*}}$
	satisfies
	\begin{equation}
		0<\lambda_{1}\le\lambda_{2}\le...\lim_{p\to\infty}\lambda_{p}=\infty.\label{eigen2}
	\end{equation}
	A related fractional elliptic equation with homogeneous source term, i.e, $G=0$ in Eqs \eqref{problem}-\eqref{condition}
	has been introduced  in section 4.2 in  \cite{jin} where  the authors  established the ill-posedness  of the problem in the sense of Hadamard \cite{hada}.
 This means that a solution of Problem  \eqref{problem}-\eqref{condition}  corresponding to the data does not always exist, and in the case of existence, it
	does not depend continuously on the given data. In fact, from
	small noise contaminated physical
	measurements, the corresponding solutions will have large
	errors. Hence, one has to resort to a
	regularization.  In \cite{jin}, the authors did not mention the regularization results for this problem.

	{  { If we replace the operator $\mathcal{A}$ by $-\mathcal{A}$ in equation \eqref{problem} then we get the fractional wave equation which is studied in \cite{yamamoto}. }As introduced in \cite{yamamoto}, the kinds of the equation \eqref{problem} have many applications in anamolous diffusion phenomenon and in heterogeneous media.  Some more physical applications can be found in \cite{yamamoto}.   }

	Until now, to the best of  our knowledge, there are no results concerning  a regularization  for  the   nonlinear problem \eqref{problem}-\eqref{condition}.
	Motivated by this reason, in this paper, we study the regularization results for \eqref{problem}-\eqref{condition}. 	In addition, one usually meets the measurement in practice, i.e. we
	need to assume the presence of an approximation $\left(\mathbf{u}_{0}^{\ep},\mathbf{u}_{1}^{\ep}\right)\in L^{2}\left(\Omega_{2}\right)\times L^{2}\left(\Omega_{2}\right)$. If the errors are generated from uncontrollable
	sources (or called external reason) as environment, wind, rain, humidity, etc, then the model is random. As we know, the problem with random data is more  difficult than the deterministic case. Hence, we  study  the problem \eqref{problem}-\eqref{condition} with the  following  random model
	\begin{equation}
	\mathbf{u}_{0}^{\ep}(y)=\mathbf{u}_{0}(y)+\ep \xi(y) ,~~	\mathbf{u}_{1}^{\ep}(y)=\mathbf{u}_{1}(y)+\ep \xi(y) \label{obs11111}
	\end{equation}
	in which the constant $\varepsilon>0$ represents the upper bound
	of the noise level in $L^{2}\left(\Omega_{2}\right)$.	And $\xi$
	is a Gaussian white noise process.
	In practice, we only obtain finite errors as follows
	\begin{equation}
	\left \langle		\mathbf{u}_{0}^{\ep} , \phi_p \right \rangle = \left \langle		\mathbf{u}_{0} , \phi_p \right \rangle  +\ep \left \langle	\xi , \phi_p \right \rangle ,\quad \left \langle		\mathbf{u}_{1}^{\ep} , \phi_p \right \rangle = \left \langle		\mathbf{u}_{1} , \phi_j \right \rangle  +\ep \left \langle	\xi , \phi_p \right \rangle, \quad  p=\overline {1,{\bf N}}.  \label{obs222222}
	\end{equation}
	where ${\bf N} $  is the  natural number which is the number of steps for discrete observations. Our task here is to find a regularized solution (called the estimator) ${\bf u}_{\text re}$ for ${\bf u}$
	and then investigate the rate of
	convergence ${\bf E} \|{\bf u}_{\text re} -{\bf u}  \| $
	, which is called the mean integrated square error (MISE). Here ${\bf E}$
	denotes the expectation w.r.t. the distribution of the data
	in the model \eqref{obs11111}.

	If $G=0$ in Eqs \eqref{problem} and ${\bf u_1}=0$ in \eqref{condition}, using   \eqref{mildsolution}, we can see that the  solution to \eqref{problem}-\eqref{condition}   satisfies a following linear operator with   random noise defined in \eqref{obs11111}
	{
\begin{equation}
		\mathcal 	K {\bf u}(a,y)+\text{"random noise"}= {\bf u}_0(y), \label{linear}
	\end{equation}

	where $ \mathcal K v=\sum_{p=1}^{\infty} \frac{1} {E_{\beta,1} \left(\la_p a^\beta\right)}\left\langle v,\phi_{p}\right\rangle $.} The  linear random model \eqref{obs11111}-\eqref{linear} is one of many linear inverse problems  in  statistics which  have been studied by
	well-known methods  including spectral cut-off (or called truncation method)  \cite{Bi2,Cavalier,Mair,Hohage},  the Tiknonov method \cite{Cox}, iterative regularization methods \cite{Engl}.  For the nonlinear problem, we can not transform  \eqref{problem}-\eqref{condition} into \eqref{linear}. Hence, previous techniques for solving \eqref{linear} are not suitable for solving the nonlinear problem \eqref{problem}-\eqref{condition}.
The main idea in this paper  is to approximate
the initial  data
$\left(\mathbf{u}_{0},\mathbf{u}_{1}\right)$
by an  approximate
data and use this function to establish a solution of a regularized problem by truncation method.


This paper is organized as follows. In section 2, we present a mild solution and show the ill-posedness of the solution to fractional semilinear elliptic equation. In section 3, we establish a regularized solution and investigate the convergence rates of the expectation of  the difference for  the solution and the regularized solution in  $L^2$ and in the Sobolev spaces $H^q$ for $q>0$.

	\section{The mild solution of  Cauchy problem for fractional elliptic equation}

	Suppose that problem \eqref{problem}-\eqref{condition} has  a mild solution ${\bf u}$ which has the form
	$
	{\bf u}(t,y)= \sum_{p=1}^\infty {\bf u}_p(t) \phi_p(y).
	$
	Then  the function ${\bf u}_p(t)$ solves  the following ordinary differential equation
	\begin{equation}
		\label{problem01}
		\left\{\begin{array}{l l l}
			\frac{D^{\beta}\mathbf{u}_p\left(t\right)}{Dt^{\beta}} -  \la_p   {\bf u}_p(t) & = 	\left\langle G(t,y, {\bf u}(t,\cdot)),\phi_p \right\rangle,\\
			{\bf u}_p(0) & = \left\langle {\bf u}_0,\phi_p\right\rangle\\
			\frac{d}{dt} {\bf u}_p(0) & = \left\langle {\bf u_1},\phi_p\right\rangle\\
		\end{array}
		\right.
	\end{equation}
	By applying the method in \cite{yamamoto, yamamoto1}, we obtain	the solution of \eqref{problem01} as follows
	\begin{align}
		{\bf u}_p(t)&= E_{\beta,1}(\lambda_p t^\beta)
		\left\langle {\bf u}_0,\phi_p\right\rangle+ t E_{\beta,2}(\lambda_p t^\beta)
		\left\langle {\bf u_1},\phi_p\right\rangle\nn\\
		&+ \int_0^t (t-\eta)^{\beta-1} E_{\beta,\beta}
		(\lambda_p (t-\eta)^\beta) 	\left\langle G(t,\eta, {\bf u}(t,\cdot)),\phi_p\right\rangle d\eta
	\end{align}
	and ${\bf u}$ is given by
	\begin{align} \label{mildsolution}
		{\bf u}(t,y) &= \sum_{p=1}^\infty \Big[ E_{\beta,1}(\lambda_p t^\beta)
		\left\langle{\bf u_0},\phi_p \right\rangle+ t E_{\beta,2}(\lambda_p t^\beta)
		\left\langle {\bf u_1},\phi_p\right\rangle \Big]\phi_p(y) \nn\\
		&+ \sum_{p=1}^\infty \Big[\int_0^t (t-\eta)^{\beta-1} E_{\beta,\beta}
		(\lambda_p (t-\eta)^\beta) 	\left\langle G(t,\eta, {\bf u}(t,\cdot)), \phi_p\right\rangle d\eta \Big]\phi_p(y)
	\end{align}
	
	Next we give some  lemmas that will be useful in this paper.
	
	\begin{lemma} \label{lem1}
		Let $0< \beta_0 <\beta_1<2 $ and $\beta \in [\beta_0, \beta_1]$. Then for $z \in \mathbb R, z \ge 0$ then
		\begin{equation}
			\frac{\widetilde  C }{\beta} e^{z^{\frac{1}{\beta}}} \le E_{\beta, 1}(z)	 \le \frac{\overline   C }{\beta} e^{z^{\frac{1}{\beta}}}.
		\end{equation}
	\end{lemma}
	\begin{proof}
		The proof can be found in \cite{Trong}.
	\end{proof}
	Now, we have the following Lemma
	\begin{lemma} \label{qt}
		 Let $0<\beta<2$ and $t \in [0,a]$. Then there exists $C_1, C_2, C_3$ which does not depend on $t$,	such that
		\begin{align}
			E_{\beta,1} (\la_p t^\beta ) \le C_1 \exp \big( \la_p^{\frac{1}{\beta} }t   \big) \label{bdt1} \\
			t	E_{\beta,2} (\la_p t^\beta ) \le  C_2  \Big( 1+ \la_p^{\frac{-1}{\beta}}  \Big) \exp \Big( \la_p^{\frac{1}{\beta}} t \Big) \label{bdt2} \\
			t^{\beta-1}	E_{\beta,\beta} (\la_p t^\beta ) \le C_3 \exp \Big( \la_p^{\frac{1}{\beta}} t \Big) \label{bdt3}
		\end{align}
	\end{lemma}
	
	\begin{proof}
		Applying Proposition 2.5 in \cite{Milos},  we obtain
		\begin{align} \label{bdt1}
			E_{\beta, \gamma} (w t^\beta) \le C_{\beta, \gamma} \Big( 1+ w^{\frac{1-\gamma}{\beta}}  \Big)  \Big(  1+ t^{1-\gamma}\Big)  \exp \Big( w^{\frac{1}{\beta}} t \Big),~~w \ge 0, t \ge 0.
		\end{align}
		Let $w= \la_p  $ and $\gamma=1$ into \eqref{bdt1}, we get
		\begin{align}
			E_{\beta,1} (\la_p t^\beta )  \le 4C_{\beta, \gamma} \exp \Big( w^{\frac{1}{\beta}} t \Big)= C_1 \exp \Big( \la_p^{\frac{1}{\beta}} x \Big) .
		\end{align}
		Let $w= \la_p  $ and $\gamma=2$ into \eqref{bdt1}, we get
		\begin{align}
			E_{\beta,2} (\la_p t^\beta )  \le C_{\beta, \gamma} \Big( 1+ \la_p^{\frac{-1}{\beta}}  \Big)  \Big(  1+ t^{-1}\Big)  \exp \Big( \la_p^{\frac{1}{\beta}} t \Big).
		\end{align}
		Multiplying both sides of the latter inequality with  $x$, we obtain
		\begin{align}
			x 	E_{\beta,2} (\la_p t^\beta )  \le C_{\beta, \gamma} \Big( 1+ \la_p^{\frac{-1}{\beta}}  \Big)  \Big(  1+ a\Big)  \exp \Big( w^{\frac{1}{\beta}} t \Big)= C_2  \Big( 1+ \la_p^{\frac{-1}{\beta}}  \Big) \exp \Big(\la_p^{\frac{1}{\beta}} t \Big).
		\end{align}
		Let $w= \la_p  $ and $\gamma=\beta$ into \eqref{bdt1}, we get
		\begin{align}
			E_{\beta,\beta} (\la_p t^\beta )  \le C_{\beta, \gamma} \Big( 1+ \la_p^{\frac{1-\beta}{\beta}}  \Big)  \Big(  1+ t^{1-\beta}\Big)  \exp \Big( \la_p^{\frac{1}{\beta}} t \Big).
		\end{align}
		Multipying bothsides of the latter inequality to $t^{\beta-1}$ and noting that $\beta>1$, we obtain
		\begin{align}
			t^{\beta-1}	E_{\beta,\beta} (\la_p t^\beta )  &\le C_{\beta, \gamma} \Big( 1+ \la_p^{\frac{1-\beta}{\beta}}  \Big)  \Big(  1+ t^{\beta-1}\Big)  \exp \Big( w^{\frac{1}{\beta}} t \Big) \nn\\
			&\le \underbrace{ C_{\beta, \gamma}	\Big(  1+ a^{\beta-1}\Big)  \Big( 1+ \la_1^{\frac{1-\beta}{\beta}}  \Big)}_{C_3}   \exp \Big( \la_p^{\frac{1}{\beta}} t \Big)  .
		\end{align}
	\end{proof}
	
	\subsection{The ill-posedness of problem \eqref{problem}-\eqref{condition}  with random noise}

	In this section, we show that the problem (1.1)-(1.2) in a special case  with random noise is ill-posed in the sense of Hadamard.


	\begin{theorem} \label{theorem2.1}
		Problem  \eqref{problem}-\eqref{condition} is ill-posed in the sense of Hadamard.
	\end{theorem}
	\begin{proof}
		Now, we give an example which shows that 	Problem  \eqref{problem}-\eqref{condition} has a unique solution and its solution is not stable.   For simple computation, we assume that  { $ \Omega_2 =(0,  \pi)$ }, {  $\mathcal{A}=-\Delta $  where $\Delta$ is the Laplacian operator,} and the function ${\bf u}_1=0$.   It  immediately follows that  $\la_{\bf N}={\bf N}^2 $.
		
		Let us consider the following parabolic equation
		\bq
		\left\{ \begin{gathered}
			\frac{D^{\beta}\mathbf{V}_{\bf N(\ep)}\left(t,y\right)}{dt^{\beta}}=\mathcal{A} \mathbf{V}_{\bf N(\ep)} \left(t,y\right)+ \overline G \left(t,y, \mathbf{V}_{\bf N(\ep)}\left(t,y\right)\right),\quad\left(t,y\right)\in\Omega:=\Omega_{1}\times\Omega_{2} \hfill \\
			\mathbf{V}_{\bf N(\ep)}(t,0)=\mathbf{V}_{\bf N(\ep)}(t,\pi)=0,\hfill \\
			\mathbf{V}_{\bf N(\ep)}(0,y)={\bf U}_{\bf N
				(\ep)}(y),\quad  \frac{d \mathbf{V}_{\bf N(\ep)}(0,y) }{dt}=0  \hfill\\
		\end{gathered}  \right. \label{ex1}
		\eq	
		where $\overline G$ is  given by
		\begin{equation}
			\overline G (t,y, v(t,y))= \sum_{p=1}^\infty  \frac{   \exp \big( \la_p^{\frac{1}{\beta} } (t-a)   \big)  }  {2a C_3  } \left< v(t,\cdot), \phi_p \right> \phi_p(y)
		\end{equation}
		for any $v \in L^2(\Omega_2)$,
		and $\phi_p(y)= \sqrt{\frac{2}{\pi}} \sin (py)$ and $C_3$ is defined in Lemma \eqref{qt}.
		Let  	 $  {\bf U}_{\bf N
			(\ep)} \in \mathcal   L^2(\Omega_2)$ be such that
		\begin{equation}
			{\bf U}_{\bf N
				(\ep)}(y)=  \sum_{p=1}^{\bf N(\ep) }  	\left \langle	{\bf u}_0^\ep , \phi_p \right \rangle  \phi_p(y)
		\end{equation}
		where $ {\bf u}_0^\ep$
		is defined	by
		\begin{equation}
			\left \langle  {\bf u}_0^\ep, \phi_j \right \rangle =\ep  \left \langle	 \xi , \phi_j \right \rangle ,\quad j=\overline {1,{\bf N}(\ep) }.  \label{obs33333}
		\end{equation}
		By
		the usual MISE (mean integrated squared error) decomposition   which involves a variance term and a bias term (see p.9, \cite{Mai}),  we get
		\begin{align}
			{\bf E}	\|  {\bf U}_{\bf N
				(\ep)}   \|_{L^2(\Omega)}^2  &={\bf E}	 \Big( \sum_{j=1}^{{\bf N(\delta) }}  	\left \langle {\bf u}_0^\ep , \phi_j \right \rangle^2  \Big) =  \ep^2  {\bf E} \Big( \sum_{j=1}^{{\bf N(\ep) }}  \xi_j^2 \Big)= \ep^2 {\bf N(\ep) }.
		\end{align}	
		The solution of Problem \eqref{ex1} is given by Fourier series
		\begin{align}
			&	\mathbf{V}_{\bf N(\ep)}(t,y) \nn\\
			& =\sum_{p=1}^{\infty}\left[ E_{\beta,1} \left(\la_p t^\beta\right)\left\langle {\bf U}_{\bf N
				(\ep)},\phi_{p}\right\rangle +\int_{0}^{t} (t-\eta)^{\beta-1}   E_{\beta,\beta} \left(\la_p (t-\eta)^\beta\right)  \left\langle \overline G\left(\eta,\cdot,\mathbf{V}_{\bf N(\ep)} \left(\eta,\cdot\right)\right),\phi_{p}\right\rangle d\eta \right]\phi_{p}\left(y\right).\label{exexact}
		\end{align}
		We show that Problem \eqref{exexact} has a unique solution $\mathbf{V}_{\bf N(\ep)} \in C([0,a]; L^2(\Omega_2))$. Let us consider
		\begin{align}
			&\mathcal H  v:	\nn\\
			& =\sum_{p=1}^{\infty}\left[ E_{\beta,1} \left(\la_p t^\beta\right)\left\langle  {\bf U}_{\bf N
				(\ep)} ,\phi_{p}\right\rangle +\int_{0}^{t} (t-\eta)^{\beta-1}   E_{\beta,\beta} \left(\la_p (t-\eta)^\beta\right)  \left\langle \overline G\left(\eta,\cdot,v \left(\eta,\cdot\right)\right),\phi_{p}\right\rangle d\eta \right]\phi_{p}\left(y\right).
		\end{align}
		For any $v_1, v_2 \in  C([0,a]; L^2(\Omega_2)) $, using H\"older inequality and Lemma \eqref{qt}, we have for all $t \in [0,a]$
		\begin{align}
			\|\mathcal H v_1(t)- \mathcal H v_2(t)  \|^2&=\sum_{p=1}^{\infty} \left[\int_{0}^{t} (t-\eta)^{\beta-1}   E_{\beta,\beta} \left(\la_p (t-\eta)^\beta\right)  \left\langle \overline G\left(\eta,\cdot,v_1 \left(\eta,\cdot\right)\right)-\overline G\left(\eta,\cdot,v_2 \left(\eta,\cdot\right)\right),\phi_{p}\right\rangle d\eta\right]^2 \nn\\
			&\le a \sum_{ p=1}^\infty \int_{0}^{t} \Big| (t-\eta)^{\beta-1}   E_{\beta,\beta} \left(\la_p (t-\eta)^\beta\right)  \Big|^2\Big| \left\langle \overline G\left(\eta,\cdot,v_1 \left(\xi,\cdot\right)\right)-\overline G\left(\eta,\cdot,v_2 \left(\eta,\cdot\right)\right),\phi_{p}\right\rangle^2 d\eta \nn\\
			&\le  \frac{1}{4a}  \sum_{ p=1}^\infty \int_0^t  \exp \big(2 \la_p^{\frac{1}{\beta} }(t-a)   \big)  \Big\langle	 	v_1(\eta)	 -  	v_2(\eta)	  , \phi_p  \Big\rangle^2  d\eta \nn\\
			&\le   \frac{1}{4} \|v_1-v_2\|_{ C([0,a]; L^2(\Omega_2))}^2.
		\end{align}		
		Hence, we obtain that
		\begin{align}
			\|	\mathcal H  v_1- \mathcal H  v_2  \||_{ C([0,a]; L^2(\Omega_2))} \le \frac{1}{2} \|v_1-v_2\|_{ C([0,a]; L^2(\Omega_2))}.
		\end{align}
		This implies that $\mathcal H$  is a contraction. Using
		the Banach fixed-point theorem, we conclude that the equation $\mathcal H (w)=w$ has a
		unique solution  $	\mathbf{V}_{\bf N(\ep)} \in C([0,a]; L^2(\Omega_2) )$.	
		Using the inequality $a^2+b^2 \ge \frac{1}{2} (a-b)^2,~~a, b \in \mathbb{R}$, 	we have the following estimate
		\begin{align}
			\Big\| \mathbf{V}_{\bf N(\ep)} \Big\|_{L^2(\Omega_2) }^2 & \ge  \underbrace{\frac{1}{2}	\Big\|  \sum_{ p=1}^\infty E_{\beta,1} \left(\la_p t^\beta\right)\left\langle  {\bf U}_{\bf N
					(\ep)} ,\phi_{p}\right\rangle  \phi_p(y) \Big\|_{L^2(\Omega_2) }^2}_{I_1} \nn\\
			& -  \underbrace {	\Big\| \sum_{p=1}^{\infty}\left[ \int_{0}^{t} (t-\eta)^{\beta-1}   E_{\beta,\beta} \left(\la_p (t-\eta)^\beta\right)  \left\langle \overline G\left(\eta,\cdot,\mathbf{V}_{\bf N(\ep)} \left(\eta,\cdot\right)\right),\phi_{p}\right\rangle d\eta\right]\phi_{p}\left(y\right)\Big\|_{L^2(\Omega_2) }^2  }_{I_2}. \label{es1}
		\end{align}
		First, using H\"older's inequality and Lemma \eqref{qt}, we get
		\begin{align}
			I_2 &=\sum_{p=1}^{\infty} \left[\int_{0}^{t} (t-\xi)^{\beta-1}   E_{\beta,\beta} \left(\la_p (t-\eta)^\beta\right)  \left\langle \overline G\left(\eta,\cdot,\mathbf{V}_{\bf N(\ep)} \left(\eta,\cdot\right)\right),\phi_{p}\right\rangle d\eta\right]^2 \nn\\
			&\le a \sum_{ p=1}^\infty \int_{0}^{t} \Big| (t-\eta)^{\beta-1}   E_{\beta,\beta} \left(\la_p (t-\eta)^\beta\right)  \Big|^2  \left\langle \overline G\left(\eta,\cdot,\mathbf{V}_{\bf N(\ep)} \left(\eta,\cdot\right)\right),\phi_{p}\right\rangle^2 d\eta  \nn\\
			&\le  \frac{1}{4a}  \sum_{ p=1}^\infty \int_0^t  \exp \big(2 \la_p^{\frac{1}{\beta} }(t-a)   \big)  \Big\langle	 	\mathbf{V}_{\bf N(\ep)}	   , \phi_p  \Big\rangle^2  d\eta\nn\\
			&\le   \frac{1}{4} \|\mathbf{V}_{\bf N(\ep)}\|_{ C([0,a]; L^2(\Omega_2))}^2. \label{es2}
		\end{align}
		And using Lemma \ref{lem1}, we have the lower bound for $I_1$ as follows
		\begin{align}
			{\bf  E}	I_1 &= \frac{1}{2}  \sum_{ p=1}^\infty \Big|E_{\beta,1} \left(\la_p t^\beta\right)\Big|^2 \Big| {\bf E} \left\langle  {\bf U}_{\bf N
				(\ep)} ,\phi_{p}\right\rangle^2  \nn\\
			&= \frac{1}{2 } \sum_{p=1}^{\bf N(\ep)} \ep^2 	\Big|E_{\beta,1} \left(\la_p t^\beta\right)\Big|^2 \ge \frac{\widetilde  C }{2 \beta } \ep^2  \exp\Big (2t |\la_{\bf N(\ep) }|^{\frac{1}{\beta}} \Big). \label{es3}
		\end{align}
		Combining \eqref{es1}, \eqref{es2}, \eqref{es3}, we obtain
		\begin{align}  \label{es111}
			{\bf  E}	\Big\|  \mathbf{V}_{\bf N(\ep)}  \Big\|_{ L^2(\Omega_2) }^2 + \frac{1}{4} 	{\bf  E} \|\mathbf{V}_{\bf N(\ep)}\|_{ C([0,a]; L^2(\Omega_2))}^2 \ge \frac{\widetilde  C }{2 \beta }  \ep^2  \exp\Big (2t | {\bf N(\ep) }|^{\frac{2}{\beta}} \Big).
		\end{align}
		By taking supremum of  both sides on $[0,a]$, we get
		\begin{align} \label{es112}
			{\bf  E}	 \|\mathbf{V}_{\bf N(\ep)}\|_{ C([0,a]; L^2(\Omega_2))}^2  \ge \frac{2\overline C}{5}  \sup_{0 \le t \le a}  \ep^2  \exp\Big (2t | {\bf N(\ep) }|^{\frac{2}{\beta}} \Big) 	=\frac{2\widetilde  C \ep^2}{5\beta} \exp\Big (2a | {\bf N(\ep) }|^{\frac{2}{\beta}} \Big) .
		\end{align}
		Let us choose  ${\bf N}:={\bf N}(\ep)=\Big[ \Big(\frac{2}{a} \ln (\frac{1}{\ep}) \Big)^{\frac{\beta}{2}}\Big]+1$, where $[z]$ is the greatest integer less than or equal to $z$. Then using  \eqref{es111},  we  obtain
		\begin{align} \label{ob1}
			{\bf E}	\| 	{\bf U}_{\bf N
				(\ep)}\|_{L^2(\Omega_2)}^2    = \ep^2 {\bf N(\ep) } \le  \ep^2 	\Big(\frac{2}{a} \ln (\frac{1}{\ep}) \Big)^{\frac{\beta}{2}}+\ep^2 \to 0,~\text{when}~\ep \to 0.
		\end{align}
		and by \eqref{es112}, we get
		\begin{align} \label{ob2}
			{\bf  E}	 \|\mathbf{V}_{\bf N(\ep)}\|_{ C([0,a]; L^2(\Omega_2))}^2   \ge 	 	\frac{2\widetilde  C}{5\beta \ep^2}  \to +\infty,~\text{when}~\ep  \to 0.
		\end{align}	
		From \eqref{ob1} and \eqref{ob2},  {  the expectation of  input data $\mathbf{U}_{\bf N(\ep)}$ tends to zero, while  the expectation of  output data $\mathbf{V}_{\bf N(\ep)}$ tends to infinity. Hence, we can conclude that Problem  \eqref{problem}-\eqref{condition}  is ill-posed in the sense  of Hadamard.  }
	\end{proof}
	
	\section{Regularization and error estimate}
	
	Next we prove the following lemma
	\begin{lemma} \label{lemmawhitenoise}
		Let $\overline U_{ {\bf N}(\ep) }^0,~\overline U_{ {\bf N}(\ep) }^1 \in L^2(\Omega_2) $  be such that
		\begin{equation}
			\overline U_{ {\bf N}(\ep) }^0(y) = \sum_{p=1}^{\bf N(\ep) } 	\left \langle		{\bf u}_0^\ep , \phi_p \right \rangle  \phi_p(y),~~	\overline U_{ {\bf N}(\ep) }^1(y) = \sum_{p=1}^{\bf N(\ep) } 	\left \langle		{\bf u}_1^\ep , \phi_p \right \rangle  \phi_p(y)
		\end{equation}
		Suppose that $ {\bf u}_0,~{\bf u}_1 \in H^{2\gamma}(\Omega_2) $. Then we have the following estimates
		\begin{align}
			&{\bf E}	\| \overline U_{ {\bf N}(\ep) }^0 -{\bf u}_0 \|_{L^2(\Omega_2)}^2  \le \ep^2 {\bf N}(\ep)  + \frac{1}{\la_{\bf N(\ep)}^{2\gamma}} \|{\bf u}_0\|_{H^{2\gamma}(\Omega_2)}^2\nn\\
				&{\bf E}	\| \overline U_{ {\bf N}(\ep) }^1 -{\bf u}_1 \|_{L^2(\Omega_2)}^2  \le \ep^2 {\bf N}(\ep)  + \frac{1}{\la_{\bf N(\ep)}^{2\gamma}} \|{\bf u}_1\|_{H^{2\gamma}(\Omega_2)}^2
		\end{align}
		for any $\gamma \ge 0$. Here ${\bf N}$ depends on $\ep$ and satisfies that $\lim_{\ep \to 0} {\bf N}(\ep)  =+\infty$ and  {  $\lim_{\ep \to 0} \ep^2{\bf N}(\ep) =0 . $ }
	\end{lemma}

	\begin{proof}
		For the
		following proof, we consider the genuine model (\ref{obs222222}). By
		the usual MISE decomposition which involves a variance term and a
		bias term, we get
		\begin{align}
			{\bf E}	\| \overline U_{ {\bf N}(\ep) }^0 -{\bf u}_0 \|_{L^2(\Omega_2)}^2  &={\bf E}	 \Big( \sum_{p=1}^{{\bf N(\ep) }}  	\left \langle 	\mathbf{u}_{0}^{\ep}-\mathbf{u}_{0} , \phi_p \right \rangle^2  \Big)+\sum_{p \geq {\bf N(\ep) }+1} 	\left \langle	{\bf u}_0 , \phi_p \right \rangle^2 \nn\\
			&=  \ep^2  {\bf E} \Big( \sum_{p=1}^{{\bf N(\ep) }}  \xi_j^2 \Big)+ \sum_{p \geq {\bf N(\ep) }+1} \la_p^{-2\gamma}  \la_p^{2\gamma}	\left \langle	{\bf u}_0 , \phi_p \right \rangle^2 .
		\end{align}
		Since $\xi_{j} \stackrel {iid}{\sim} N(0,1)$, it follows that  $ {\bf E} \xi_j^2=1$, so  the proof is completed.
	\end{proof}
		In this paper, we apply the truncation method to establish a regularized solution as follows
		\begin{align}
		\mathbf{u}_{\bf N(\ep)}^\ep\left(t,y\right) & =\sum_{p=1}^{\infty}\mathcal {R} (\la_p, {\bf N(\ep)}) \left[ E_{\beta,1} \left(\la_p t^\beta\right)\left\langle \overline U_{ {\bf N}(\ep) }^0,\phi_{p}\right\rangle +t E_{\beta,2} \left(\la_p t^\beta\right)\left\langle \overline U_{ {\bf N}(\ep) }^1,\phi_{p}\right\rangle \right.\nonumber \\
		& \left.+\int_{0}^{t} (t-\eta)^{\beta-1}   E_{\beta,\beta} \left(\la_p (t-\eta)^\beta\right)  \left\langle G\left(\eta,\cdot,\mathbf{u}_{\bf N(\ep)}^\ep\left(\eta,\cdot\right)\right),\phi_{p}\right\rangle d\eta \right]\phi_{p}\left(y\right),\quad\left(t,y\right)\in\Omega.\label{eq:regu}
		\end{align}
		Here $\mathcal {R} (\la_p, {\bf N(\ep)})=1 $ if $\la_p \le B_{\bf N(\ep)}$ and is zero if $\la_p > B_{\bf N(\ep)}$ and $B_{\bf N(\ep)}$ is called a parameter of regularization which will be chosen later.

	Our  main result is as follows
	\begin{theorem} \label{theorem3.1}
		The integral equation \eqref{eq:regu} has a unique solution ${\bf u}^\ep_{\bf N (\ep) } \in C([0,a];L^2(\Omega_2))$. Suppose that ${\bf u}_0,{\bf u}_1 \in H^\gamma(\Omega_2) $ that  satisfy
		\[
		\|{\bf u}_0\|_{H^{2\gamma}(\Omega_2)}+\|{\bf u}_1\|_{H^{2\gamma}(\Omega_2)} \le \mathcal M_0.
		\]
		Assume that problem \eqref{problem}-\eqref{condition}  has a unique mild solution ${\bf u}$ which satisfies that
		\begin{equation}
			\sum_{ p=1 }^\infty  \la_p ^{ \mu } \exp \Big(2(a-t)\la_p^{\frac{1}{\beta}} \Big)  \left \langle	{\bf u}(t,\cdot) , \phi_p \right \rangle^2 \le \mathcal M,~~t \in [0,a],
		\end{equation}
		for some   positive constants $\mu, \mathcal M$.
		Assume that   $B_{\bf N(\ep) }$ satisfy
		\begin{equation} \label{cond}
			\lim_{\ep \to 0}  B_{\bf N(\ep)}= +\infty,~	\lim_{\ep \to 0}  \exp \big( 2 |B_{\bf N(\ep)}|^{\frac{1}{\beta} }a   \big)\ep^2 {\bf N}(\ep) = \lim_{\ep \to 0} \frac{	 \exp \big( 2 |B_{\bf N(\ep)}|^{\frac{1}{\beta} }a   \big)}{\la_{\bf N(\ep)}^{2\gamma}}=0.
		\end{equation}
		Then the following estimate holds
		\begin{align}
			&	{\bf E}	\| 		\mathbf{u}_{\bf N(\ep)}^\ep\left(t,.\right)-	\mathbf{u}\left(t,.\right) \|_{L^2(\Omega_2)}^2\nn\\
			&\le 2 C_1 \exp \big( 2|B_{\bf N(\ep)}|^{\frac{1}{\beta} } t   \big) \Bigg( 2\ep^2 {\bf N}(\ep)  + \frac{\mathcal M_0}{\la_{\bf N(\ep)}^{2\gamma}} \Bigg)+2D_1 \exp \Big(-2(a-t) |B_{\bf N(\ep)}|^{\frac{1}{\beta} }  \Big) \la_{\bf N(\ep)}^{-\mu} \mathcal M^2.\nn\\
		\end{align}
	\end{theorem}
	\begin{remark}
	 	From the theorem above, it is easy to see that
$	{\bf E}	\| 		\mathbf{u}_{\bf N(\ep)}^\ep\left(t,.\right)-	\mathbf{u}\left(t,.\right) \|_{L^2(\Omega_2)}^2 $ is of order {
\begin{equation}
\max\bigg[\la_{\bf N(\ep)}^{-\mu}\exp \Big(-2(a-t) |B_{\bf N(\ep)}|^{\frac{1}{\beta} }  \Big) , \ep^2 {\bf N}(\ep)  e^{2a  |B_{\bf N(\ep)}|^{\frac{1}{\beta} }  } , \frac{ e^{2a  |B_{\bf N(\ep)}|^{\frac{1}{\beta} }  } }{\la_{\bf N(\ep)}^{2\gamma}}  \Big].
\end{equation}}
	We give one example for the choice of ${\bf N} (\ep)$
	which satisfies the condition \eqref{cond}.  It is well-known that $\la_{\bf N(\ep)} \sim (\bf N(\ep))^{\frac{2}{d}} $, we can choose ${\bf N} (\ep)$ such that
		$
		{\bf N(\ep)}= [\ep^{\frac{-2b}{2m+1}}]
		$ for some $b>0$ and
	\[
	e^{k a |B_{\bf N(\ep)}|^{\frac{1}{\beta} }  }=({\bf N(\ep)})^m,~~0<m < \frac{2\gamma}{d}.
	\]
	Then, we get
	\[
	B_{\bf N(\ep)}= \left(\frac{m}{ka} \log ( {\bf N(\ep)} ) \right)^\beta
		\]
	Then the error $	{\bf E}	\| 		\mathbf{u}_{\bf N(\ep)}^\ep\left(t,.\right)-	\mathbf{u}\left(t,.\right) \|_{L^2(\Omega_2)}^2 $ is of order
	\begin{equation}
\ep^{\frac{ 4bm(a-x)}{ (2m+1)a } }  \max \Big( \ep^{2-2b},  \ep^{\frac{2b(4\gamma-2md)}{(2m+1)d } }, \ep^{   \frac{4b \mu}{ (2m+1)d }  }    \Big).
	\end{equation}
	
	\end{remark}
	
	\begin{proof}[{\bf Proof of Theorem \ref{theorem3.1}}]
		We divide the proof into some smaller parts.\\
		\noindent 	{\bf Part  1.} The existence and uniqueness  of the solution of the nonlinear integral equation \eqref{eq:regu} .
		
		For $v \in C( [0,a]; L^2(\Omega_2))) $, we put
		\begin{align}
			\mathcal{F}(v)(t,y) & =\sum_{p=1}^{\infty}\mathcal {R} (\la_p, {\bf N}(\ep)) \bigg[ E_{\beta,1} \left(\la_p t^\beta\right)\left\langle \overline U_{ {\bf N}(\ep) }^0,\phi_{p}\right\rangle +t E_{\beta,2} \left(\la_p t^\beta\right)\left\langle \overline U_{ {\bf N}(\ep) }^1,\phi_{p}\right\rangle \nonumber \\
			& +\int_{0}^{t} (t-\eta)^{\beta-1}   E_{\beta,\beta} \left(\la_p (t-\eta)^\beta\right)  \left\langle G\left(\eta,\cdot,v\left(\eta,\cdot\right)\right),\phi_{p}\right\rangle d\eta \bigg]\phi_{p}\left(y\right).
		\end{align}
		We will  prove by induction  that if  $v_1, v_2 \in C( [0,a]; L^2(\Omega_2))) $ then
		\begin{align} \label{induction}
			&\Big\|  \mathcal{F}^m(w_1)(t,.)-     \mathcal{F}^m(w_2)(t,.) \Big\|_{ L^2(\Omega_2) } \nn\\
			&\le   \Bigg(  \frac{K^2 a^2A_1^2\la_1^{\frac{2-2\beta}{\beta} }  \exp \Big( 2|B_{\bf N(\ep)}|^{\frac{1}{\beta} } a  \Big) }{\beta^2}\Bigg)^m \frac{t^m}{m!} \|w_1-w_2\|_{C( [0,a]; L^2(\Omega_2))}.
		\end{align}
		For $m=1$, we have by using Lemma \ref{qt} and the fact that $G$ is Lipchitz
		\begin{align}
			&\|  \mathcal{F}(v_1)-  \mathcal{F}(v_2) \|_{L^2(\Omega_2)}^2  \nn\\
			&	\le t  \sum_{p=1}^{\infty} |\mathcal {R} (\la_p, {\bf N}(\ep))|^2  \int_{0}^{t} (t-\eta)^{2\beta-2}  | E_{\beta,\beta} \left(\la_p (t-\eta)^\beta\right)|^2  \left\langle G\left(\eta,\cdot,v_1\left(\eta,\cdot\right)\right)-G\left(\eta,\cdot,v_2\left(\eta,\cdot\right)\right),\phi_{p}\right\rangle^2 d\eta \nn\\
			&\le 	\frac{a^2A_1^2}{\beta^2} \la_1^{\frac{2-2\beta}{\beta} } \int_0^t \exp \Big( 2 |B_{\bf N(\ep)}|^{\frac{1}{\beta} } (t-\eta)   \Big)  \Big\| G\left(\eta,\cdot,v_1\left(\eta,\cdot\right)\right)-G\left(\eta,\cdot,v_2\left(\eta,\cdot\right)\right)\Big\|^2_{L^2(\Omega_2)}d\eta\nn\\
			& \le \frac{K^2 a^2A_1^2\la_1^{\frac{2-2\beta}{\beta} } t}{\beta^2}   \exp \Big( 2 |B_{\bf N(\ep)}|^{\frac{1}{\beta} } a  \Big)  \Big\|v_1-v_2\Big\|^2_{C([0,a];L^2(\Omega_2))}.
		\end{align}
		Assume that \eqref{induction} holds for $m=p$. We show that  \eqref{induction} holds for $m=p+1$. In fact, we have
		\begin{align}
			&\|  \mathcal{F}^{p+1}(v_1)-  \mathcal{F}^{p+1}(v_2) \|_{L^2(\Omega_2)}^2  \nn\\
			&	\le t  \sum_{p=1}^{\infty} |\mathcal {R} (\la_p, {\bf N}(\ep) )|^2  \int_{0}^{t} (t-\eta)^{2\beta-2}  | E_{\beta,\beta} \left(\la_p (t-\eta)^\beta\right)|^2  \left\langle G\left(\eta,\cdot, \mathcal{F}^{p}(v_1)\left(\eta,\cdot\right)\right)-G\left(\eta,\cdot, \mathcal{F}^{p}(v_2)\left(\eta,\cdot\right)\right),\phi_{p}\right\rangle^2 d\eta \nn\\
			&\le 	\frac{a^2A_1^2\la_1^{\frac{2-2\beta}{\beta} }}{\beta^2}  \int_0^t \exp \Big( 2 |B_{\bf N(\ep)}|^{\frac{1}{\beta} } (t-\eta)   \Big)  \Big\| G\left(\eta,\cdot,\mathcal{F}^{p}(v_1)\left(\eta,\cdot\right)\right)-G\left(\eta,\cdot,\mathcal{F}^{p}(v_2)\left(\eta,\cdot\right)\right)\Big\|^2_{L^2(\Omega_1)}
d\eta \nn\\
			& \le \frac{K^2 a^2A_1^2\la_1^{\frac{2-2\beta}{\beta} } t }{\beta^2}   \exp \Big( 2 |B_{\bf N(\ep)}|^{\frac{1}{\beta} } a  \Big)  \Big\|\mathcal{F}^{p}(v_1)-\mathcal{F}^{p}(v_2)\Big\|^2_{C([0,a];L^2(\Omega_2))}\nn\\
			&\le  \Bigg(  \frac{K^2 a^2A_1^2\la_1^{\frac{2-2\beta}{\beta} }  \exp \Big( 2\la_{\bf N(\ep)}^{\frac{1}{\beta} } a  \Big) }{\beta^2}\Bigg)^{p+1} \frac{x^{p+1}}{(p+1)!} \|v_1-v_2\|_{C( [0,a]; L^2(\Omega_2))}.
		\end{align}
		Therefore, by  induction, we have \eqref{induction} for all $w, v \in C([0,a];L^2(\Omega_2))$. Since
		$$
		\lim_{m \to +\infty} \Bigg(  \frac{K^2 a^2A_1^2\la_1^{\frac{2-2\beta}{\beta} }  \exp \Big( 2|B_{\bf N(\ep)}|^{\frac{1}{\beta} } a  \Big) }{\beta^2}\Bigg)^m \frac{a^m}{m!} =0
		$$
		there exists a positive integer  $m_0$ such that $\mathcal{F}^{m_0}$ is a contraction. It follows that the
		equation $\mathcal{F}^{m_0}w=w$ has a unique solution $u_{N(\ep)}^\ep \in C([0,a];L^2(\Omega_2))$. We claim that
		$
		\mathcal{F} ( u_{N(\ep)}^\ep ) = u_{N(\ep)}^\ep .
		$
		In fact, since $\mathcal{F}^{m_0} ( u_{N(\ep)}^\ep) = u_{N(\ep)}^\ep $, we know that $ \mathcal {F} \left( \mathcal{F}^{m_0} ( u_{N(\ep)}^\ep )  \right)=\mathcal{F} ( u_{N(\ep)}^\ep ) $. This is equavilent to $   \mathcal {F}^{m_0} \left( \mathcal{F} ( u_{N(\ep)}^\ep )  \right)=\mathcal{F} ( u_{N(\ep)}^\ep )  $. Hence, $\mathcal{F} ( u_{N(\ep)}^\ep ) $ is a fixed point of $\mathcal {F}^{m_0}$. Moreover, as noted above, $u_{N(\ep)}^\ep$ is a fixed point of $ \mathcal {F}^{m_0} $.
		
		{\bf Part 2}. Estimate the expectation of the error between the exact solution ${\bf u}$ and the regularized solution $	\mathbf{u}_{\bf N(\ep)}^\ep $.\\
		Let us consider the following integral equation
		\begin{align}
			\mathbf{v}_{\bf N(\ep)}^\ep\left(t,y\right) & =\sum_{p=1}^{\infty}\mathcal {R} (\la_p, {\bf N}) \left[ E_{\beta,1} \left(\la_p t^\beta\right)\left\langle \mathbf{u}_{0},\phi_{p}\right\rangle +t E_{\beta,2} \left(\la_p t^\beta\right)\left\langle \mathbf{u}_{1},\phi_{p}\right\rangle \right.\nonumber \\
			& \left.+\int_{0}^{t} (t-\eta)^{\beta-1}   E_{\beta,\beta} \left(\la_p (x-\eta)^\beta\right)  \left\langle G\left(\eta,\cdot,\mathbf{v}_{\bf N(\ep)}^\ep\left(\eta,\cdot\right)\right),\phi_{p}\right\rangle d\eta \right]\phi_{p}\left(y\right),\quad\left(t,y\right)\in\Omega,\label{vep}
		\end{align}
	Combining \eqref{eq:regu} and \eqref{vep} and	taking the expectation of both sides of the norm in $L^2$, we get
		\begin{align}
			&{\bf E}	\| 		\mathbf{u}_{\bf N(\ep)}^\ep\left(t,.\right)-	\mathbf{v}_{\bf N(\ep)}^\ep\left(t,.\right) \|_{L^2(\Omega_2)}^2 \nn\\
			&\le 3  	{\bf E}\Bigg(  \sum_{ \la_p \le B_{\bf N(\ep)} } |E_{\beta,1} \left(\la_p t^\beta\right)|^2  \left \langle	\overline U_{ {\bf N}(\ep) }^0- {\bf u}_0 , \phi_p \right \rangle^2\Bigg)\nn\\
			&+3  	{\bf E}\Bigg(  \sum_{ \la_p \le B_{\bf N(\ep)} } |tE_{\beta,2} \left(\la_p t^\beta\right)|^2  \left \langle	\overline U_{ {\bf N}(\ep) }^1- {\bf u}_1 , \phi_p \right \rangle^2\Bigg)\nn\\
			&+  3  	{\bf E}\Bigg(  \sum_{ \la_p \le B_{\bf N(\ep)} }  \left[ \int_{0}^{t} (t-\eta)^{\beta-1}   E_{\beta,\beta} \left(\la_p (t-\eta)^\beta\right)  \left\langle G\left(\eta,\cdot,\mathbf{u}_{\bf N(\ep)}^\ep\left(\eta,\cdot\right)\right)- G\left(\eta,\cdot,\mathbf{v}_{\bf N(\ep)}^\ep\left(\eta,\cdot\right)\right),\phi_{p}\right\rangle d\eta \right]^2\Bigg).
		\end{align}
		Where above  we have used the inequality $(a+b+c)^2 \le 3a^2+ 3b^2+3c^2$ for real numbers $a,b,c.$
		Using Lema \ref{lemmawhitenoise} and the H\"older inequality, we deduce that
		\begin{align}
			&{\bf E}	\| 		\mathbf{u}_{\bf N(\ep)}^\ep\left(t,.\right)-	\mathbf{v}_{\bf N(\ep)}^\ep\left(t,.\right) \|_{L^2(\Omega_2)}^2\nn\\
			  &\le \frac{ 3 A_1^2}{\beta^2} \exp \big( 2  |B_{\bf N(\ep)}|^{\frac{1}{\beta}}  t   \big)	{\bf E}		\| \overline U_{ {\bf N}(\ep) }^0 -{\bf u}_0 \|_{L^2(\Omega_2)}^2 \nn\\
			&+ \frac{3A_1^2}{\beta^2} \la_1^{-\frac{2}{\beta}} \exp \big( 2  |B_{\bf N(\ep)}|^{\frac{1}{\beta}}  t   \big)	{\bf E}		\| \overline U_{ {\bf N}(\ep) }^1 -{\bf u}_1 \|_{L^2(\Omega_2)}^2 \nn\\
			&+  \frac{3k^2 a A_1^2}{\beta^2} \la_1^{\frac{2-2\beta}{\beta}}  \int_0^t \exp \Big( 2  |B_{\bf N(\ep)}|^{\frac{1}{\beta}}  (t-\eta )   \Big)		{\bf E}	\| 		\mathbf{u}_{\bf N(\ep)}^\ep\left(\eta,.\right)-	\mathbf{v}_{\bf N(\ep)}^\ep\left(\eta,.\right) \|_{L^2(\Omega_2)}^2  d\eta.
		\end{align}
		Multiplying both sides with  $\exp \big(- 2  |B_{\bf N(\ep)}|^{\frac{1}{\beta}} t   \big)$, we obtain
		\begin{align}
			&\exp \big(- 2  |B_{\bf N(\ep)}|^{\frac{1}{\beta}}  t   \big)	{\bf E}	\| 		\mathbf{u}_{\bf N(\ep)}^\ep\left(t,.\right)-	\mathbf{v}_{\bf N(\ep)}^\ep\left(t,.\right) \|_{L^2(\Omega_2)}^2 \nn\\
			&\le  \frac{ 3 A_1^2}{\beta^2}	{\bf E}		\| \overline U_{ {\bf N}(\ep) }^0 -{\bf u}_0 \|_{L^2(\Omega_y)}^2 +\frac{3A_1^2}{\beta^2} \la_1^{-\frac{2}{\beta}}	{\bf E}		\| \overline U_{ {\bf N}(\ep) }^1 -{\bf u}_1 \|_{L^2(\Omega_2)}^2 \nn\\
			&	+  \frac{3k^2 a A_1^2}{\beta^2} \la_1^{\frac{2-2\beta}{\beta}}  \int_0^t \exp \Big( -2  |B_{\bf N(\ep)}|^{\frac{1}{\beta}} \eta   \Big)		{\bf E}	\| 		\mathbf{u}_{\bf N(\ep)}^\ep\left(\eta,.\right)-	\mathbf{v}_{\bf N(\ep)}^\ep\left(\eta,.\right) \|_{L^2(\Omega_2)}^2  d\eta.
		\end{align}	
		Applying Gronwall's inequality, we get
		\begin{align}
			&\exp \big(- 2  |B_{\bf N(\ep)}|^{\frac{1}{\beta}}  t  \big)	{\bf E}	\| 		\mathbf{u}_{\bf N(\ep)}^\ep\left(t,.\right)-	\mathbf{v}_{\bf N(\ep)}^\ep\left(t,.\right) \|_{L^2(\Omega_2)}^2  \nn\\
			&\le \frac{3A_1^2}{\beta^2} \max \big(1, \la_1^{\frac{2-2\beta}{\beta}} \big)\exp \Big( \frac{3k^2 a A_1^2}{\beta^2} \la_1^{\frac{2-2\beta}{\beta}} \Big)  \Bigg(  	{\bf E}		\| \overline U_{ {\bf N}(\ep) }^0 -{\bf u}_0 \|_{L^2(\Omega_2)}^2+	{\bf E}		\| \overline U_{ {\bf N}(\ep) }^1 -{\bf u}_1 \|_{L^2(\Omega_2)}^2 \Bigg)\nn\\
			&\le \underbrace{\frac{3A_1^2}{\beta^2} \max \big(1, \la_1^{\frac{2-2\beta}{\beta}} \big)\exp \Big( \frac{3k^2 a A_1^2}{\beta^2} \la_1^{\frac{2-2\beta}{\beta}} \Big)}_{C_1:=C_1(\beta, A_1,a,k, \la_1)} \Bigg( 2\ep^2 {\bf N}(\ep)  + \frac{\|{\bf u}_0\|^2_{H^{2\gamma}(\Omega_2)}+\|{\bf u}_1\|^2_{H^{2\gamma}(\Omega_y)}}{\la_{\bf N(\ep)}^{2\gamma}} \Bigg)\nn\\
			&\le C_1 \Bigg( 2\ep^2 {\bf N}(\ep)  + \frac{\mathcal M_0}{\la_{\bf N(\ep)}^{2\gamma}} \Bigg) . \label{saiso1}
		\end{align}
		Now, we continue to estimate $\| {\bf u}(t,.)- \mathbf{v}_{\bf N(\ep)}^\ep\left(t,.\right)\|_{L^2(\Omega_2)}$. Indeed, using H\"older inequality, globally Lipschitzp roperty of $G$, and equations \eqref{mildsolution} and \eqref{exexact}  we get
		\begin{eqnarray}
			\begin{aligned}
				&\| {\bf u}(t,.)- \mathbf{v}_{\bf N(\ep)}^\ep\left(t,.\right)\|_{L^2(\Omega_2)}^2\nn\\
				& \le 2  \sum_{ \la_p \le B_{\bf N(\ep)} }  \left[ \int_{0}^{t} (t-\eta)^{\beta-1}   E_{\beta,\beta} \left(\la_p (t-\eta)^\beta\right)  \left\langle G\left(\eta,\cdot,\mathbf{u} \left(\eta,\cdot\right)\right)- G\left(\eta,\cdot,\mathbf{v}_{\bf N(\ep)}^\ep\left(\eta,\cdot\right)\right),\phi_{p}\right\rangle d\eta \right]^2 \nn\\
				&+2 \sum_{ \la_p > B_{\bf N(\ep)} } \left \langle	{\bf u}(t,y) , \phi_p \right \rangle^2 \nn\\
				&\le 2 \sum_{ \la_p > B_{\bf N(\ep)} }  \la_p ^{- \mu } \exp \Big(-2(a-t)\la_p^{\frac{1}{\beta}} \Big) \la_p ^{ \mu } \Big(2(a-t)\la_p^{\frac{1}{\beta}} \Big)  \left \langle	{\bf u}(t,y) , \phi_p \right \rangle^2  \nn\\
				&+\frac{2k^2 a A_1^2}{\beta^2} \la_1^{\frac{2-2\beta}{\beta}}  \int_0^t \exp \Big( 2 B_{\bf N(\ep)}^{\frac{1}{\al} }(t-\eta)   \Big)		\big\| 		\mathbf{u}\left(\eta,.\right)-	\mathbf{v}_{\bf N(\ep)}^\ep\left(\eta,.\right) \big\|_{L^2(\Omega_2)}^2  d\eta \\
				&\le |B_{\bf N(\ep)}|^{-\mu}\exp \Big(-2(a-t) |B_{\bf N(\ep)}|^{\frac{1}{\beta}} \Big)  \mathcal M^2\nn\\
				&+ \frac{2k^2 a A_1^2}{\beta^2} \la_1^{\frac{2-2\beta}{\beta}}  \int_0^t \exp \Big( 2 |B_{\bf N(\ep)}|^{\frac{1}{\beta} }(t-\eta)   \Big)		\big\| 		\mathbf{u}\left(\eta,.\right)-	\mathbf{v}_{\bf N(\ep)}^\ep\left(\eta,.\right) \big\|_{L^2(\Omega_2)}^2  d\eta.
			\end{aligned}
		\end{eqnarray}
		Multiplying both sides with  $\exp \Big(2(a-t)|B_{\bf N(\ep)}|^{\frac{1}{\beta}} \Big)$, we obtain
		\begin{align}
			&\exp \Big(2(a-t) |B_{\bf N(\ep)}|^{\frac{1}{\beta}} \Big)	\| {\bf u}(t,.)- \mathbf{v}_{\bf N(\ep)}^\ep\left(t,.\right)\|_{L^2(\Omega_2)}^2  \nn\\
			&~~~\le|B_{\bf N(\ep)}|^{-\mu} \mathcal M^2+\frac{2k^2 a A_1^2}{\beta^2} \la_1^{\frac{2-2\beta}{\beta}}  \int_0^t \exp \Big( 2 |B_{\bf N(\ep)}|^{\frac{1}{\beta} }(a-\eta)   \Big)		\big\| 		\mathbf{u}\left(\eta,.\right)-	\mathbf{v}_{\bf N(\ep)}^\ep\left(\eta,.\right) \big\|_{L^2(\Omega_2)}^2  d\eta.
		\end{align}
		Gronwall's inequality implies that
		\begin{equation}
			\exp \Big(2(a-t) |B_{\bf N(\ep)}|^{\frac{1}{\beta}} \Big)	\| {\bf u}(t,.)- \mathbf{v}_{\bf N(\ep)}^\ep\left(t,.\right)\|_{L^2(\Omega_2)}^2  \le  \underbrace{\exp\Big( \frac{2k^2 a A_1^2t}{\beta^2} \la_1^{\frac{2-2\beta}{\beta}}  \Big)}_{D_1:=D_1(k,a,A_1,\beta)} |B_{\bf N(\ep)}|^{-\mu} \mathcal M^2.
		\end{equation}
		This together with the estimate \eqref{saiso1} leads to
		\begin{align}
			&	{\bf E}	\| 		\mathbf{u}_{\bf N(\ep)}^\ep\left(t,.\right)-	\mathbf{u}\left(t,.\right) \|_{L^2(\Omega_2)}^2\nn\\  &\le 	2 	{\bf E}	\| 		\mathbf{u}_{\bf N(\ep)}^\ep\left(t,.\right)-	\mathbf{v}_{\bf N(\ep)}^\ep\left(t,.\right) \|_{L^2(\Omega_2)}^2 + 2 \| {\bf u}(t,.)- \mathbf{v}_{\bf N(\ep)}^\ep\left(t,.\right)\|_{L^2(\Omega_2)} \nn\\
			&\le 2 C_1 \exp \big( 2 |B_{\bf N(\ep)}|^{\frac{1}{\beta} }t   \big) \Bigg( 2\ep^2 {\bf N}(\ep)  + \frac{\mathcal M_0}{\la_{\bf N(\ep)}^{2\gamma}} \Bigg)+2D_1 \exp \Big(-2(a-t)  |B_{\bf N(\ep)}|^{\frac{1}{\beta}} \Big) |B_{\bf N(\ep)}|^{-\mu} \mathcal M^2\nn\\
		\end{align}
		which completes our proof.
	\end{proof}
	
	The next result provides an error estimate in the Sobolev space $H^q (\Omega_2)$ which is equipped with a
	norm defined by
	\bes
	\|g\|_{H^q(\Omega_2)}^2=  \sum\limits_{p=1}^{\infty}    \la_p^q  \Big<g, \phi_{p}\Big>^{2}.
	\ens
	To estimate the error in the  $H^q$ norm, we need stronger assumption on  solution ${\bf u}$.
	\begin{theorem}
		Suppose that the  problem \eqref{problem}-\eqref{condition}  has unique solution ${\bf u}$ such that
		\begin{equation} \label{assumption2}
			\sum_{ p=1 }^\infty   \exp \Big(2(a-t+r)\la_p^{\frac{1}{\beta}} \Big)  \left \langle	{\bf u}(t,y) , \phi_p \right \rangle^2 \le \mathcal M_1,~~t \in [0,a],
		\end{equation}
		for any $r>0$.
		Let   ${\bf N}(\ep), B_{\bf N(\ep)}$ be as in Theorem \eqref{theorem3.1}.
		Then the following estimate holds
		\begin{align}
			&	{\bf E}\|  	\mathbf{u}_{\bf N(\ep)}^\ep\left(t,.\right)- {\bf u}(t,.)\|_{H^q(\Omega_2)}^2\nn\\
			&\le 4 |B_{\bf N (\ep)}|^q \exp \big( 2 |B_{\bf N(\ep)}|^{\frac{1}{\beta} }t   \big) C_1 \Bigg( 2\ep^2 {\bf N}(\ep)  + \frac{\mathcal M_0}{\la_{\bf N(\ep)}^{2\gamma}} \Bigg)+\mathcal M_1^2 (2D_1+1)|B_{\bf N (\ep)}|^q \exp\Big(-2(a-t+r) |B_{\bf N(\ep)} |^{\frac{1}{\beta}} \Big)
		\end{align}
	\end{theorem}
{
\begin{remark}
	In physical modelling and engineering, the estimation on a Hilbert scale  space, for example $H^q(\Omega)$ is important. Furthermore, the problem of estimating the error in this space more difficult than $L^2(\Omega)$. Hence, the above theorem is a new and interesting result.
\end{remark}	
	
}
	
	\begin{proof}
		First, we have
		\begin{align}
			{\bf E}\|  	\mathbf{u}_{\bf N(\ep)}^\ep\left(t,.\right)-{\mathcal  Q}_{B_{\bf N (\ep)}} {\bf u}(t,.)\|_{H^q(\Omega_2)}^2&= {\bf E} \left( \sum_{ \la_p \le B_{\bf N (\ep) } } \la_j^{q}  \left \langle \mathbf{u}_{\bf N(\ep)}^\ep\left(t,.\right)- {\bf u}(t,.) , \phi_p(y) \right \rangle^2  \right)\nn\\
			& \le |B_{\bf N (\ep)}|^q {\bf E} \left( \sum_{ \la_p \le B_{\bf N (\ep) } }  \left \langle \mathbf{u}_{\bf N(\ep)}^\ep\left(t,.\right)- {\bf u}(t,.) , \phi_p (y) \right \rangle^2  \right)\nn\\
			&\le |B_{\bf N (\ep)}|^q  	{\bf E}\|  	\mathbf{u}_{\bf N(\ep)}^\ep\left(t,.\right)- {\bf u}(t,.)\|_{L^2(\Omega_2)}^2. \label{er1}
		\end{align}
{ where  ${\mathcal  Q}_{B_{\bf N (\ep)}} {\bf u}(t,.)= \sum_{ \la_p \le B_{\bf N (\ep) } }    \left \langle  {\bf u}(t,.) , \phi_p(y) \right \rangle \phi_p (y) $}.
		Under the assumption \eqref{assumption2}, we get
		\begin{eqnarray}
			\begin{aligned}
				&\| {\bf u}(t,.)- \mathbf{v}_{\bf N(\ep)}^\ep\left(t,.\right)\|_{L^2(\Omega_2)}^2\nn\\
				& \le 2  \sum_{ \la_p \le B_{\bf N(\ep)} }  \left[ \int_{0}^{t} (t-\eta)^{\beta-1}   E_{\beta,\beta} \left(\la_p (t-\eta)^\beta\right)  \left\langle G\left(\eta,\cdot,\mathbf{u} \left(\eta,\cdot\right)\right)- G\left(\eta,\cdot,\mathbf{v}_{\bf N(\ep)}^\ep\left(\eta,\cdot\right)\right),\phi_{p}\right\rangle d\eta\right]^2\nn\\
				& +2 \sum_{ \la_p > B_{\bf N(\ep)} } \left \langle	{\bf u}(t,y) , \phi_p \right \rangle^2 \nn\\
				&\le 2 \sum_{ \la_p > B_{\bf N(\ep)} }   \exp \Big(-2(a-t+r)\la_p^{\frac{1}{\beta}} \Big)  \exp \Big(2(a-t+r)\la_p^{\frac{1}{\beta}} \Big)  \left \langle	{\bf u}(t,y) , \phi_p \right \rangle^2  \nn\\
				&+\frac{2k^2 a A_1^2}{\beta^2} \la_1^{\frac{2-2\beta}{\beta}}  \int_0^t \exp \Big( 2 |B_{\bf N(\ep)}|^{\frac{1}{\beta}}  (t-\eta)   \Big)		\big\| 		\mathbf{u}\left(\eta,.\right)-	\mathbf{v}_{\bf N(\ep)}^\ep\left(\eta,.\right) \big\|_{L^2(\Omega_2)}^2  d\eta \\
				&\le \exp \Big(-2(a-t+r) |B_{\bf N(\ep)}|^{\frac{1}{\beta}} \Big)  \mathcal M_1^2\nn\\
				&+ \frac{2k^2 a A_1^2}{\beta^2} \la_1^{\frac{2-2\beta}{\beta}}  \int_0^t \exp \Big( 2   |B_{\bf N(\ep)}|^{\frac{1}{\beta}} (t-\eta)   \Big)		\big\| 		\mathbf{u}\left(\eta,.\right)-	\mathbf{v}_{\bf N(\ep)}^\ep\left(\eta,.\right) \big\|_{L^2(\Omega_2)}^2  d\eta
			\end{aligned}
		\end{eqnarray}
		Multiplying both sides with $\exp \Big(2(a-t)  |B_{\bf N(\ep)}|^{\frac{1}{\beta}} \Big)$, we obtain
		\begin{align}
			&\exp \Big(2(a-t)  |B_{\bf N(\ep)}|^{\frac{1}{\beta}} \Big)	\| {\bf u}(t,.)- \mathbf{v}_{\bf N(\ep)}^\ep\left(t,.\right)\|_{L^2(\Omega_2)}^2  \nn\\
			&~\le \exp \Big(-2r  |B_{\bf N(\ep)}|^{\frac{1}{\beta}} \Big)  \mathcal M_1^2 \nn\\
			&+\frac{2k^2 a A_1^2}{\beta^2} \la_1^{\frac{2-2\beta}{\beta}}  \int_0^t \exp \Big( 2  |B_{\bf N(\ep)}|^{\frac{1}{\beta}} (a-\eta)   \Big)		\big\| 		\mathbf{u}\left(\eta,.\right)-	\mathbf{v}_{\bf N(\ep)}^\ep\left(\eta,.\right) \big\|_{L^2(\Omega_2)}^2  d\eta.
		\end{align}
		Then Gronwall's inequality implies that
		\begin{equation}
		\exp \Big(2(a-t)  |B_{\bf N(\ep)}|^{\frac{1}{\beta}} \Big)	\| {\bf u}(t,.)- \mathbf{v}_{\bf N(\ep)}^\ep\left(t,.\right)\|_{L^2(\Omega_2)}^2  \le   D_1 \exp \Big(-2r  |B_{\bf N(\ep)}|^{\frac{1}{\beta}}  \Big)  \mathcal M_1^2
		\end{equation}
		This latter estimate together with the estimate \eqref{saiso1} leads to
		\begin{align}
			&	{\bf E}	\| 		\mathbf{u}_{\bf N(\ep)}^\ep\left(t,.\right)-	\mathbf{u}\left(t,.\right) \|_{L^2(\Omega_2)}^2\nn\\  &\le 	2 	{\bf E}	\| 		\mathbf{u}_{\bf N(\ep)}^\ep\left(t,.\right)-	\mathbf{v}_{\bf N(\ep)}^\ep\left(t,.\right) \|_{L^2(\Omega_2)}^2 + 2 \| {\bf u}(t,.)- \mathbf{v}_{\bf N(\ep)}^\ep\left(t,.\right)\|_{L^2(\Omega_2)} \nn\\
			&\le  \exp \big( 2   |B_{\bf N(\ep)}|^{\frac{1}{\beta}} t   \big) \Bigg[2 C_1 \Bigg( 2\ep^2 {\bf N}(\ep)  + \frac{\mathcal M_0}{\la_{\bf N(\ep)}^{2\gamma}} \Bigg)+2 D_1 \exp \Big(-2(r+a) |B_{\bf N(\ep)}|^{\frac{1}{\beta}}  \Big)  \mathcal M_1^2\Bigg]. \label{er2}
		\end{align}
		It  follows from \eqref{er1} that
		\begin{align}
			&{\bf E}\|  	\mathbf{u}_{\bf N(\ep)}^\ep\left(t,.\right)-{\mathcal  Q}_{B_{\bf N (\ep)}} {\bf u}(t,.)\|_{H^q(\Omega_2)}^2\nn\\
			& \le  |B_{\bf N (\ep)}|^q \exp \big( 2 B_{\bf N(\ep)}^{\frac{1}{\al} }t   \big) \Bigg[2 C_1 \Bigg( 2\ep^2 {\bf N}(\ep)  + \frac{\mathcal M_0}{\la_{\bf N(\ep)}^{2\gamma}} \Bigg)+2 D_1 \exp \Big(-2(r+a)\la_{\bf N(\ep)}^{\frac{1}{\beta}} \Big)  \mathcal M_1^2\Bigg].
		\end{align}
		On the other hand, consider  the function
		\begin{equation}
			\mathcal G(z)= z^q e^{-D z},~~D>0, \label{ine1}
		\end{equation}
		From the derivative of $\mathcal G$ is $\mathcal G'(z)=z^{q-1} e^{-Dz} (q-Dz) $, we know that
		$\mathcal G $ is strictly decreasing when $Dz \ge q$.   Since $\lim_{\ep \to 0} B_{\bf N(\ep)} =+\infty $, we see that if $\ep$ small enough then  $2r B_{\bf N(\ep)}  \ge q$.  Replacing  $D=2(a-t+r),~z=B_{\bf N(\ep)}  $ into \eqref{ine1}, we obtain
		for $\la_p > B_{\bf N(\ep)}  $
		\be
		\mathcal G(\la_p)=	\la_p^{q}  \exp\Big(-2(a-t+r) \la_p^{\frac{1}{\beta}} \Big)  \le \mathcal G(B_{\bf N(\ep)} )=  |B_{\bf N (\ep)}|^q \exp\Big(-2(a-t+r) |B_{\bf N(\ep)} |^{\frac{1}{\beta}} \Big)
		\en
		The latter equality leads to
		\begin{align}
			\|  	\mathbf{u}\left(t,.\right)-{\mathcal  Q}_{B_{\bf N (\ep)}} {\bf u}(t,.)\|_{H^q(\Omega_2)}^2&=  \sum_{ \la_p > B_{\bf N (\ep) } } \la_p^{q}  \left \langle  {\bf u}(t,y) , \phi_p(y) \right \rangle^2 \nn\\
			&= \sum_{ \la_p > B_{\bf N (\ep) } } 	\mathcal G(\la_p) \exp\Big(2(a-t+r) \la_p^{\frac{1}{\beta}} \Big)  \left \langle  {\bf u}(t,y) , \phi_p(y) \right \rangle^2 \nn\\
			&\le \mathcal G(B_{\bf N(\ep)} )  \sum_{ \la_p > B_{\bf N (\ep) } } 	 \exp\Big(2(a-t+r) \la_p^{\frac{1}{\beta}} \Big)  \left \langle  {\bf u}(t,y) , \phi_p(y) \right \rangle^2  \nn\\
			&\le \mathcal M_1^2 |B_{\bf N (\ep)}|^q \exp\Big(-2(a-t+r) |B_{\bf N(\ep)} |^{\frac{1}{\beta}} \Big)  \label{er3}
		\end{align}
		where we use  the assumption \eqref{assumption2} for  the latter inequality.
		Combining \eqref{er1}, \eqref{er2} and \eqref{er3}, we deduce that
		\begin{align}
			&	{\bf E}\|  	\mathbf{u}_{\bf N(\ep)}^\ep\left(t,.\right)- {\bf u}(t,.)\|_{H^q(\Omega_2)}^2\nn\\
			&\le 2	{\bf E}\|  	\mathbf{u}_{\bf N(\ep)}^\ep\left(t,.\right)-{\mathcal  Q}_{B_{\bf N (\ep)}} {\bf u}(t,.)\|_{H^q(\Omega_2)}^2 +2	\|  	\mathbf{u}\left(t,.\right)-{\mathcal  Q}_{B_{\bf N (\ep)}} {\bf u}(t,.)\|_{H^q(\Omega_2)}^2 \nn\\
			&\le 4 |B_{\bf N (\ep)}|^q \exp \big( 2  |B_{\bf N(\ep)}|^{\frac{1}{\beta}} t   \big) C_1 \Bigg( 2\ep^2 {\bf N}(\ep)  + \frac{\mathcal M_0}{\la_{\bf N(\ep)}^{2\gamma}} \Bigg)+\mathcal M_1^2 (2D_1+1)|B_{\bf N (\ep)}|^q \exp\Big(-2(a-t+r) |B_{\bf N(\ep)} |^{\frac{1}{\beta}} \Big)
		\end{align}
		which completes the proof.
	\end{proof}


\begin{thebibliography}{99}
		\bibitem{Bi2} 	 N. Bissantz,  H. Holzmann. \emph{ Asymptotics for spectral regularization estimators in statistical inverse problems} Comput. Statist. 28 (2013), no. 2, 435--453.
		
		
		
		
		
		\bibitem {Cavalier} L. Cavalier. \emph{ Nonparametric statistical inverse problems} Inverse Problems 24 (2008), no. 3, 034004, 19 pp.
		
		\bibitem{Cox} D. D. Cox. \emph{Approximation of method of regularization estimators,} Ann. Statist., 16 (1988),
		pp. 694--712.
		
		

		
		
		
		\bibitem{Trong}	D. T. Dang, E. Nane, D. M. Nguyen and N. H. Tuan. \emph{ Continuity of solutions of a class of fractional equations} Potential Anal. To Appear, 2017.
		\bibitem {Engl} H. W. Engl, M. Hanke, and A. Neubauer. \emph{ Regularization of Inverse Problems,} Kluwer
		Academic, Dordrecht, Boston, London, 1996.
		
		
\bibitem{hada} J. Hadamard. \emph {Lectures on the Cauchy Problem in Linear Differential Equations, } Yale University Press, New Haven, CT, 1923.


\bibitem {jin} B. Jin,  W. Rundell. \emph{A tutorial on inverse problems for anomalous diffusion processes}  Inverse Problems 31 (2015), no. 3, 035003, 40 pp.
		
		
\bibitem{yamamoto} Y. Kian,  M. Yamamoto.  \emph{On existence and uniqueness of solutions for semilinear fractional wave equations} Fract. Calc. Appl. Anal. 20 (2017), no. 1, 117--138.
		\bibitem {Kilbas} A.A.Kilbas, H.M. Srivastava, J.J.Trujillo.  Theory and Application of Fractional differential equations, North - Holland Mathematics Studies, vol. 204, Elsevier Science B.V, Amsterdam, 2006.
		

		
\bibitem{Hohage} C. K\"onig,  F. Werner,  T. Hohage.  \emph{ Convergence rates for exponentially ill-posed inverse problems with impulsive noise.} SIAM J. Numer. Anal. 54 (2016), no. 1,  341--360.
		
		
		
		
		

\bibitem{Mai}   P.N.T.  Mai.  \emph{A statistical minimax approach to the Hausdorff moment problem},  Inverse Problems 24 (2008), no. 4, 045018, 13 pp.	
		
	

\bibitem{Mair}  A.B. Mair,  H.F. Ruymgaart.  \emph{ Statistical inverse estimation in Hilbert scales,} SIAM J. Appl. Math. 56 (1996), no. 5, 1424-1444.
		
\bibitem{mark-skorski-12} M.M. Meerschaert and A. Skosrski. Stochastic Models for Fractional Calculus. De Gruyter Studies in Mathematics, Vol. 43
Walter de Gruyter, Berlin/Boston, 2012.

\bibitem{Milos} J. Milos , R.C. Danijela.  \emph{Generalized uniformly continuous solution operators and inhomogeneous fractional evolution equations with variable coefficients,} Vol. 2017 (2017), No. 293, pp. 1--24.
		
\bibitem{podlubny-1999} I. Podlubny. \emph{Fractional Differential Equations}. Academic Press, San Diego, CA, 1999.
		\bibitem {yamamoto1} K. Sakamoto and M. Yamamoto. 	\emph{Initial value/boundary value problems for fractional diffusion-wave equations and applications to some inverse problems}  J. Math. Anal. Appl. 382 (2011), no. 1, 426--447.
		
		\bibitem{samko} S.G. Samko, A.A. Kilbas, O.I. Marichev. \emph{ Fractional Integrals and
			Derivatives: Theory and Applications.} Gordon and Breach, New York (1993).
		
		
		
		
	
		
		
		


		
		
		
		
		
	\end{thebibliography}
\end{document}